\documentclass[12pt]{article}
\usepackage{amsthm}
\usepackage{amsfonts}
\usepackage{epsfig}

\newcommand{\GG}{{\cal G}}
\newtheorem{theorem}{Theorem}
\newtheorem{conjecture}{Conjecture}

\newtheorem{lemma}[theorem]{Lemma}

\title{Classes of graphs with small rank decompositions are $\chi$-bounded}
\author{Zden\v{e}k Dvo\v{r}\'ak\thanks{Department of Applied Mathematics and Institute for Theoretical Computer Science (ITI), Faculty of Mathematics and Physics, Charles University, Malostransk{\'e} n{\'a}m.~25, 118 00 Prague, Czech Republic. E-mail: {\tt rakdver@kam.mff.cuni.cz}. Institute for Theoretical Computer Science (ITI) as project 1M0545 of the Ministry of Education of the Czech Republic.
}\and
Daniel Kr\'al'\thanks{Department of Mathematics, University of West Bohemia, Univerzitn\'i 8, 306 14 Pilsen, Czech Republic. This author is also affiliated with Department of Applied Mathematics and Institute for Theoretical Computer Science of Charles University. E-mail: {\tt dankral@kma.zcu.cz}. This author was supported by the grant GA{\v C}R P202/11/0196.}}
\date{}

\begin{document}
\maketitle

\begin{abstract}
A class of graphs $\GG$ is $\chi$-bounded
if the chromatic number of graphs in $\GG$ is bounded by a function of the clique number.
We show that if a class $\GG$ is $\chi$-bounded,
then every class of graphs admitting a decomposition along cuts of small rank
to graphs from $\GG$ is $\chi$-bounded.
As a corollary, we obtain that every class of graphs with bounded rank-width (or equivalently, clique-width)
is $\chi$-bounded.
\end{abstract}

\section{Introduction}

For a graph $G$ and an integer $k$, a \emph{proper $k$-coloring} of $G$ is a function $f:V(G)\to \{1,\ldots, k\}$ such that $f(u)\neq f(v)$
for every edge $uv\in E(G)$.  The \emph{chromatic number} $\chi(G)$ of $G$ is the smallest $k$ such that $G$ has a proper $k$-coloring.
The chromatic number is one of the most studied graph parameters, and although determining it precisely is NP-complete~\cite{garey1979computers},
a number of interesting bounds and connections to other graph parameters is known.  A natural lower bound for the chromatic number
is given by the \emph{clique number} $\omega(G)$ which is the size of the largest complete subgraph of $G$.  At first, it might
be natural to believe that there could exist an upper bound on the chromatic number in the terms of the clique number.
However, that is far from the truth.  Erd\H{o}s~\cite{erd-gircol} showed that for any $g$ there exist graphs of arbitrarily large chromatic number that do not
contain cycles of length at most~$g$.

Therefore, the graph classes where the chromatic and clique numbers are tied together are of an interest. 
The most famous example is the class of \emph{perfect graphs}
which are the graphs $G$ such that $\chi(H)=\omega(H)$ for every induced subgraph $H$ of $G$.
The Strong Perfect Graph Theorem asserts that this class can be alternatively characterized as the class of graphs $G$ such that neither $G$ nor its complement contains an induced odd
cycle of length at least $5$ (Chudnovsky et al.~\cite{sperft}).
The class of perfect graphs includes
line graphs of bipartite graphs, chordal graphs, comparability graphs and others.

For many graph classes, the connection between $\chi$ and $\omega$ is not straightforward.  For example, if $G$ is a circle-arc graph
(an intersection graph of arcs of a circle), then $\chi(G)\le 2\omega(G)$.
This leads to the following definition.
We say that
a class $\GG$ of graphs is \emph{$\chi$-bounded} if there exists a function $f$ such that $\chi(G)\le f(\omega(G))$ for every graph $G\in \GG$.
The examples of $\chi$-bounded graph classes include the circle-arc graphs, the circle graphs~\cite{kostochka1997covering},
line-graphs~\cite{vizing}, and graphs avoiding any fixed tree $T$ of radius at most two as an induced subgraph~\cite{kierstead1994radius}.
Also note that any class of graphs with bounded chromatic number is $\chi$-bounded.

Our work is motivated by the following conjecture of Geelen:
\begin{conjecture}\label{conj-vmin}
For every graph $H$, the class of graphs without a vertex-minor isomorphic to $H$ is $\chi$-bounded.
\end{conjecture}
Recall that a graph $H'$ is a {\em vertex-minor} of a graph $H$ if $H'$ can be obtained
by a series of vertex removals and neighborhood complementations. The {\em neighborhood complementation}
with respect to a vertex $v$ of a graph $G$ is the following operation: if two neighbors of $v$
are joined by an edge, delete this edge from $G$, and if they are not joined, add an edge joining
them.

A possible approach to Conjecture~\ref{conj-vmin} could consist of proving that
every graph without a vertex-minor isomorphic to $H$ admits a decomposition to well-behaved pieces
along simple-structured cuts.
Such an approach has led to proofs of several deep results in structural graph theory.
Examples include the graph minor structure
theorem, where the graphs avoiding a minor of some fixed graph are obtained by joining pieces that are (almost) embedded in surfaces of small genus along small vertex cuts~\cite{robertson2003graph}, and
the proof of the Strong Perfect Graph Theorem~\cite{sperft}.

In this paper, we introduce a new (somewhat technical) way of decomposing graphs
along cuts of small rank. This kind of decomposition will allow us to prove two special
cases of Conjecture~\ref{conj-vmin}.

\begin{theorem}
\label{thm1}
For any $k$, the class of graphs with rank-width at most $k$ is $\chi$-bounded.
\end{theorem}

\begin{theorem}
\label{thm2}
The class of graphs without a vertex-minor isomorphic to the wheel $W_5$ is $\chi$-bounded.
\end{theorem}

We derive Theorem~\ref{thm2} from a structural characterization
of graphs avoiding $W_5$ by Geelen~\cite{geelen} and another corollary of our
main result. Recall that a $1$-join of two graphs $G_1$ and $G_2$
with two distinguished vertices $v_1$ and $v_2$, respectively, is the graph obtained
by deleting the vertex $v_i$ from $G_i$, $i=1,2$, and adding an edge
between every neighbor of $v_1$ in $G_1$ and every neighbor of $v_2$ in $G_2$.
The corollary we use to derive Theorem~\ref{thm2} is the following result.

\begin{theorem}\label{thm3}
Let $\GG$ be a $\chi$-bounded class of graphs closed under taking induced subgraphs.
If $\GG'$ is the class of graphs that can be obtained from graphs of $\GG$
by repeated applications of $1$-joins,
then the class $\GG'$ is also $\chi$-bounded.
\end{theorem}

In the next section, we formally introduce the decompositions we study and
derive Theorems~\ref{thm1}--\ref{thm3} from our main result (Theorem~\ref{thm-main}).
Theorem~\ref{thm-main} is then proven in Section~\ref{sec-proof}.

\section{Decomposition along small cuts}

If we want to allow the existence of cliques of arbitrary order,
we cannot restrict the size of the cuts along those we split.
On the other hand, such cuts should not be completely arbitrary and
thus it is natural to restrict their complexity.
This leads to the following definition.
For a graph $H$ and a proper subset $W$ of vertices of $H$,
the \emph{matrix of the cut $(W,V(H)\setminus W)$} is the $|W|\times |V(H)\setminus W|$ matrix $M$
indexed by $W$ and $V(H)\setminus W$
such that $M_{uv}$ is $1$ if $uv\in E(H)$ and $0$ otherwise.
The \emph{rank} of the cut is the rank of its matrix over $F_2$.
The \emph{diversity} of the cut is the size of the largest set $S$
such that $S\subseteq W$ or $S\subseteq V(H)\setminus W$ and the vertices of $W$
have mutually distinct neighborhoods in the other side of the cut.
In other words, the diversity of the cut is the maximum number of different rows or
columns of the its matrix.
Note that if the rank of the cut is $r$, then its diversity is at most $2^r$, and,
conversely, if the diversity is $d$, then its rank is at most $d$.

A natural way of decomposing a graph is assigning its vertices to nodes of a tree.
Formally, a \emph{decomposition} of a graph $G$ is a tree $T$ with
a mapping $\tau:V(G)\to V(T)$.
Each edge $e$ of the tree $T$ naturally defines a cut in the graph $G$
with sides being the preimages of the vertex sets of the two trees obtained from $T$ by removing $e$.
The {\em rank} of a decomposition is the maximum of the ranks of the cuts
induced by its edges. Analogously, we define the {\em diversity} of a decomposition.

Every graph admits a decomposition of rank one with $T$ being a star,
so restricting the rank of decompositions is too weak. One way to
circumvent this is to restrict the structure of the decomposition. For example,
if we require that all inner nodes of $T$ have degree three and that
the vertices of $G$ are injectively mapped by $\tau$ to the leaves of $T$,
the rank of the smallest such decomposition is the {\em rank-width} of $G$,
a well-studied parameter in structural graph theory. In this paper, we proceed
in a different (more general) way.

If a tree $T$ with a mapping $\tau$ is a decomposition of a graph $G$ and
$v$ is a node of $T$, then $G_{T,v}$ is the spanning subgraph of $G$ such that
two vertices $u$ and $u'$ of $G$ are adjacent in $G_{T,v}$
if they are adjacent in $G$ and
the node $v$ lies on the unique path between $\tau(u)$ and $\tau(u')$ in $T$ (possibly,
$v$ can be $\tau(u)$ or $\tau(u')$). In other words, $G_{T,v}$
is the spanning subgraph of $G$ where we remove edges between vertices $u$ and $u'$ such that
$\tau(u)$ and $\tau(u')$ lies in the same component of $T\setminus v$.
For a class $\GG$ of graphs,
we say that the decomposition $T$ is {\em $\GG$-bounded}
if the graph $G_{T,v}$ belongs to $\GG$ for every $v\in V(T)$.

Our main result is the following theorem.
\begin{theorem}\label{thm-main}
Let $\GG$ be a $\chi$-bounded class of graphs closed under taking induced subgraphs and
$r$ an integer.
The class of graphs admitting a $\GG$-bounded decomposition with rank at most $r$ is $\chi$-bounded.
\end{theorem}

We would like to note that the definition of $\GG$-bounded decompositions can also be
extended as follows. Let $\GG_1$ and $\GG_2$ be two classes of graphs. We can require that
the subgraphs of $G$ induced by $\tau^{-1}(v)$, $v\in V(T)$, belong to $\GG_1$ and
the subgraphs $G_{T,v}\setminus \tau^{-1}(v)$ to $\GG_2$. So, the class $\GG_1$
controls the complexity of subgraphs induced by preimages of individual nodes and
the class $\GG_2$ controls the mutual interaction between different pieces of the decomposition.
However, if both $\GG_1$ and $\GG_2$ are $\chi$-bounded and closed under taking induced subgraphs,
this does not lead to a more general concept.
Indeed, let $\GG$ be the class of graphs $G$ that admit a vertex partition
$V_1$ and $V_2$ such that $G[V_i]\in\GG_i$ for $i=1,2$.
If $\GG_1$ and $\GG_2$ are $\chi$-bounded and closed under taking induced subgraphs,
so is $\GG$, and any decomposition with respect to $\GG_1$ and $\GG_2$ as
defined in this paragraph is $\GG$-bounded.

We now derive Theorems~\ref{thm1}--\ref{thm3} from Theorem~\ref{thm-main}.
The decompositions appearing in the definition of the rank-width of a graph,
which is given earlier in this section, are $\GG$-bounded where $\GG$ is the class of $3$-colorable graphs.
So, an application of Theorem~\ref{thm-main} for this class $\GG$ yields Theorem~\ref{thm1}.

The proof of Theorem~\ref{thm3} is more complicated.

\begin{proof}[Proof of Theorem~\ref{thm3}.]
Let $G$ be a graph contained in $\GG'$.
By the definition of $\GG'$, there exists a tree $T$ with nodes $v_1,\ldots,v_n$ and
graphs $G_1,\ldots,G_n$ from $\GG$ such that
\begin{itemize}
\item for every edge $v_iv_j$ of $T$, the graphs $G_i$ and $G_j$ contain
      distinguished vertices $w_{i,j}\in V(G_i)$ and $w_{j,i}\in V(G_j)$, respectively,
      and these vertices are different fordifferent choices of $v_iv_j$, and
\item the graph $G$ is obtained from the graphs $G_1,\ldots,G_n$ by $1$-joins
      with respect to vertices $w_{i,j}$ and $w_{j,i}$, $v_iv_j\in E(T)$.
\end{itemize}
Note that the order of $1$-joins does not affect the result of the procedure.

Let $\GG''$ be the class of graphs that can be obtained from $\GG$
by adding isolated vertices and blowing up some vertices to independent sets, i.e., replacing a vertex $w$
with and independent set $I$ and joining each vertex of $I$ to the neighbors of $w$.
Since $\GG$ is $\chi$-bounded and closed under taking induced subgraphs,
so is $\GG''$.
Since the tree $T$ with mapping $\tau$
that maps a vertex $u\in V(G)$ to the node $v_i$ such that $u\in V(G_i)$
is a $\GG''$-bounded decomposition of $G$ and its rank is at most one,
Theorem~\ref{thm3} now follows from Theorem~\ref{thm-main}.
\end{proof}

To derive Theorem~\ref{thm2} from Theorem~\ref{thm3},
we need the following result of Geelen~\cite{geelen}.
Observe that since the class of circle graphs is $\chi$-bounded~\cite{kostochka1997covering} and
all other basic graphs appearing in Theorem~\ref{thm-geelen}
have at most $8$ vertices, Theorem~\ref{thm2} directly follows from Theorem~\ref{thm3}.

\begin{theorem}[Geelen~\cite{geelen}, Theorem 5.14]
\label{thm-geelen}
If $G$ is a connected graph without a vertex-minor isomorphic to $W_5$, then one of the following holds:
\begin{itemize}
\item $G$ is a circle graph, or
\item $G$ is can be obtained from a graph isomorphic to $W_7$, the cube $C$ or the graph $C^-$,
that is the cube with one vertex removed, by a sequence of neighborhood complementations, or
\item there exist connected graphs $G_1$ and $G_2$ without a vertex-minor isomorphic to $W_5$ that
      have fewer vertices than $G$, such that $G$ is a $1$-join of $G_1$ and $G_2$.
\end{itemize}
\end{theorem}

\section{Proof of Main Theorem}\label{sec-proof}

In this section, we present the proof of Theorem~\ref{thm-main}.
We start with a lemma.

\begin{lemma}\label{lemma-key}
Let $d$ and $k$ be two integers and $G$ a connected graph with at least two vertices.
If $G$ has a decomposition formed by a tree $T$ and a mapping $\tau$ such that
the diversity of the decomposition is at most $d$ and $\chi(G_{T,v})\le k$ for every node $v$ of $T$,
then there exists a (not necessarily proper) coloring of the vertices of $G$
by at most $d(k+1)$ colors such that $\omega(G[\varphi^{-1}(c)])<\omega(G)$ for every color $c$.
\end{lemma}

\begin{proof}
We can assume (without loss of generality) that
$T$ has a leaf $v_1$ with $\tau^{-1}(v_1)=\emptyset$ and root $T$ at this leaf.
Fix a sequence $T_1\subseteq T_2\subseteq\cdots\subseteq T_n$ of
subtrees of $T$ such that $T_1$ is the subtree formed 
by $v_1$ solely, $T_n=T$ and
$|V(T_i)\setminus V(T_{i-1})|=1$ for every $i=2,\ldots,n$.
Let $v_i$ be the vertex of $T_i$ not contained in $T_{i-1}$.
For an edge $uu'\in E(G)$,
the \emph{origin} of $uu'$ is the nearest common ancestor of $\tau(u)$ and $\tau(u')$ in $T$.

For a node $v$ of $T$, let $T_v$ be the subtree of $T$ rooted at $v$ and
let $V_v$ be the union $\bigcup_{v'\in V(T_v)}\tau^{-1}(v')$. Since the diversity
of the decomposition given by $T$ and $\tau$ is at most $d$,
there exists a partition $V_v^0,\ldots,V_v^d$ of $V_v$ such that
the vertices of $V_v^0$ have no neighbors outside $V_v$ and
two vertices of $V_v$ belong to the same $V_v^j$ if and only if
they have the same neighbors outside of $V_v$.

The set of colors used by the constructed coloring will be $C_0=\{1,\ldots,d(k+1)\}$.
We will construct partial (not necessarily proper) colorings $\varphi_1$, $\varphi_2$, \ldots, $\varphi_n$ of $G$ 
which we use the colors $C_0$ such that
\begin{enumerate}
\item $\varphi_i$ extends $\varphi_{i-1}$ for $i=2,\ldots,m$,
\item $\varphi_i$ assigns colors to all vertices incident with edges whose origin belongs to $T_i$,
      in particular, to all vertices of $\tau^{-1}(V(T_i))$,
\item the vertices of $V_{v_{i'}}^j$ for $i'>i$ and $j=0,\ldots,d$ are either all colored with the same color or
      none of them is colored by $\varphi_i$, $i=1,\ldots,n$, and
\item if $\varphi_i(u)=\varphi_i(u')$ for an edge $uu'$ of $G$ and $i\in\{1,\ldots,n\}$,      
      then there exists $i'$, $2\le i'\le i$, a child $v'$ of $v_{i'}$, and $j\in\{1,\ldots,d\}$ such that 
      both $u$ and $u'$ are not colored by $\varphi_{i'-1}$, they are colored by $\varphi_{i'}$ and
      they both belong to $V_{v'}^j$.
\end{enumerate}

Since $v_1$ is the origin of no edge,
we can choose $\varphi_1$ as the empty coloring.
Suppose that $\varphi_{i-1}$ was already defined and let us describe the construction of $\varphi_i$
from $\varphi_{i-1}$.
Note that the vertices of $V_{v_i}$ that are not colored by $\varphi_{i-1}$ are exactly those of $V_{v_i}^0$.
Furthermore, the restriction of $\varphi_{i-1}$ to $V_{v_i}$ uses at most $d$ colors (one
for each of $V_{v_i}^1$, \ldots, $V_{v_i}^d$). Let $C$ be the set of these colors.

By the assumption of the lemma, there exists a proper coloring $\psi_1$ of the graph $G_{T,v_i}$ using $k$ colors.
Choose $\psi_1$ such that all twins, i.e., vertices with the same set of neighbors, receive the same color.
Let $W$ be the vertices of $V_{v_i}^0$ that have a non-zero degree in $G_{T,v_i}$ or are mapped by $\tau$ to $v_i$.
Let $\psi_2$ be a (not necessarily proper) coloring of the vertices of $W$ by colors $1,\ldots,d$ defined as follows:
if $\tau(w)=v_i$, the vertex $w$ is assigned the color $1$.
For $w\in W\setminus\tau^{-1}(v_i)$, let $v'$ be the child of $v_i$ such that $w\in V_{v'}$.
Set $\psi_2(w)$ to be the index $j$ such that $w\in V_{v'}^j$ (note that $j>0$).
Finally, set $\psi(w)=(\psi_1(w),\psi_2(w))$; $\psi$ is a (not necessarily proper) coloring of the vertices of $W$
with $d\cdot k$ colors.

The coloring $\varphi_i$ is an extension of $\varphi_{i-1}$ to the vertices of $W$ such that
$\varphi_i$ uses the (at least) $d\cdot k$ colors of $C_0\setminus C$ and two vertices of $W$
get the same color if and only if they are assigned the same color by $\psi$.

We now verify that $\varphi_i$ has the properties stated earlier.
\begin{enumerate}
\item This follows directly from the definition of $\varphi_i$.
\item If the origin of an edge $uu'$ of $G$ belongs to $T_i$,
      then it either belongs to $T_{i-1}$ or it is $v_i$.
      In the former case, $u$ and $u'$ are already colored by $\varphi_{i-1}$.
      In the latter case, they both belong to $W$ and are colored by $\varphi_i$.
\item If any vertex of $V_{v_{i'}}^j$ is colored by $\varphi_{i-1}$,
      then all of them are colored by $\varphi_{i-1}$ and they have the same color.
      So, suppose that none of the vertices of $V_{v_{i'}}^j$ is colored by $\varphi_{i-1}$.
      Observe that all vertices of $V_{v_{i'}}^j$ are twins in $G_{T,v_i}$.
      So, if one of them belongs to $W$, then all of them do.
      Also, they are assigned the same color by $\psi_1$.
      Let $v'$ be the child of $v_i$ such that $v_{i'}$ belongs to $T_{v'}$.
      Since the vertices of $V_{v_{i'}}^j$ have the same neighbors outside $V_{v'}$,
      they all belong to the same set $V_{v'}^{j'}$. Consequently, they are assigned
      the same color by $\psi_2$. So, they have the same color assigned by $\psi$ and
      thus by $\varphi_i$, too.
\item Suppose that $\varphi_i(u)=\varphi_i(u')$ and $uu'$ is an edge of $G$.
      If both $u$ and $u'$ are colored by $\varphi_{i-1}$,
      the property follows from the properties of $\varphi_{i-1}$.
      By symmetry, we can now assume that $u$ is not colored by $\varphi_{i-1}$.
      Since $u$ is not colored, $u'$ must belong to $V_{v_i}$.
      If $u'$ were colored by $\varphi_{i-1}$, it would belong to one of the sets $V_{v_i}^1,\ldots,V_{v_i}^d$.
      So, it would hold that $\varphi_i(u')\in C$ but $\varphi_i(u)\in C_0\setminus C$ which is impossible.
      We conclude that neither $u$ nor $u'$ is colored by $\varphi_{i-1}$.
      Since they are colored by $\varphi_i$, they both belong to $W$, and,
      since $\varphi_i$ assigns them the same color, it holds that $\psi(u)=\psi(u')$.
      The definition of $\psi_1$ implies that both $u$ and $u'$ belong to the same $V_{v'}$
      for a child $v'$ of $v_i$ (otherwise, the edge $uu'$ would be contained in $G_{T,v_i}$).
      Finally, the definition of $\psi_2$ implies that they both belong to the same $V_{v'}^j$ for $j>0$.
\end{enumerate}

We claim that $\varphi_n$ is the sought coloring. The number of colors it uses
does not exceed $d(k+1)$. Assume that $G$ contains a monochromatic maximum clique $K$.
The fourth property of the coloring $\varphi_n$ implies that all vertices of $K$
belong to $V_v^j$ for some node $v$ of $T$ and $j>0$. In particular, they have a common
neighbor which is impossible since $K$ is a maximum clique.
\end{proof}

The proof of the main theorem now follows.

\begin{proof}[Proof of Theorem~\ref{thm-main}.]
Let $f$ be the function witnessing that $\GG$ is $\chi$-bounded and
let $f'(s)=2^{rs}\prod_{i=2}^s(f(s)+1)$.
We show that $\chi(G)\le f'(\omega(G))$ for
any graph $G$ admitting a $\GG$-bounded decomposition with rank at most $r$
by induction on the clique number $\omega(G)$.

If $\omega(G)=1$, then $\chi(G)=1\le f'(1)$. Assume that $\omega(G)>1$.
By Lemma~\ref{lemma-key}, the vertices of $G$ can be colored with $2^r(f(\omega(G))+1)$ colors in such a way
that no clique of $G$ of size $\omega(G)$ is monochromatic. Fix such a coloring $\varphi$.
Consider the subgraphs of $G$ induced by the color classes of $\varphi$.
Since each of these subgraphs admits a $\GG$-bounded decomposition with rank at most $r$,
it has a proper coloring with $f'(\omega(G)-1)$ colors.
Coloring each vertex by a pair consisting of the color assigned by $\varphi$ and
the color assigned by the coloring of the corresponding color class of $G$
yields a proper coloring of $G$ with $2^r(f(\omega(G))+1)f'(\omega(G)-1)=f(\omega(G))$ colors.
\end{proof}

\section*{Acknowledgements}

The authors would like to thank DIMACS for hospitality
during the DIMACS Workshop on Graph Coloring and Structure held in Princeton, May 2009,
where they learnt about Conjecture~\ref{conj-vmin} from Jim Geelen.
They are also grateful to Sang-il Oum for bringing Theorem~\ref{thm-geelen}
to their attention as well as for helpful discussions on the subject.

\bibliographystyle{siam}
\bibliography{rankbound}

\end{document}